\documentclass[11pt]{article}

\usepackage[T1]{fontenc}
\usepackage[utf8]{inputenc}
\usepackage[margin=1in]{geometry}
\usepackage{amsmath,amssymb,amsthm,mathtools}
\usepackage{algorithm}
\usepackage[noend]{algpseudocode}
\usepackage{graphicx}
\usepackage{url}
\usepackage[colorlinks=true,linkcolor=blue,citecolor=blue,urlcolor=blue]{hyperref}

\newcommand{\R}{\mathbb{R}}
\newcommand{\ip}[2]{\left\langle #1,#2\right\rangle}
\newcommand{\norm}[1]{\left\lVert #1\right\rVert}
\newcommand{\set}[1]{\left\{#1\right\}}
\newcommand{\grad}{\nabla}
\newcommand{\Hess}{\nabla^2}
\newcommand{\dd}{\,\mathrm{d}}
\DeclareMathOperator{\argmin}{argmin}
\DeclareMathOperator{\dist}{dist}

\theoremstyle{plain}
\newtheorem{theorem}{Theorem}[section]
\newtheorem{lemma}[theorem]{Lemma}
\newtheorem{proposition}[theorem]{Proposition}

\theoremstyle{definition}
\newtheorem{definition}[theorem]{Definition}
\newtheorem{example}[theorem]{Example}
\begin{document}

\title{Continuous-Time Dynamics of the\\Difference-of-Convex Algorithm}

\author{Yi-Shuai Niu\\
\small Beijing Institute of Mathematical Sciences and Applications (BIMSA), Beijing, China\\
\small \texttt{niuyishuai@bimsa.cn}}

\date{}

\maketitle

\begin{abstract}
We study the dynamical systems structure underlying the difference-of-convex algorithm. For smooth DC decompositions with strongly convex component, we show that classical DCA admits an exact dual-coordinate interpretation as a full-step explicit Euler discretization of a nonlinear autonomous system. To pass to continuous time, we introduce a damped (under-relaxed) DCA variant and prove that its vanishing-step limit is a Hessian-Riemannian gradient flow generated by the convex component of the decomposition. For each fixed relaxation parameter, this damped DCA scheme also admits a Bregman-regularized interpretation together with monotone descent, asymptotic criticality, and KL convergence under standard boundedness assumptions. Under a metric DC-PL inequality, we further obtain a global linear rate for the damped scheme and show that the strongest provable global guarantee furnished by this estimate is attained by the half-relaxed choice, whereas local linearization near a nondegenerate minimum favors the full-step scheme. For the limiting flow, we establish an exact energy identity, asymptotic criticality of bounded trajectories, explicit global rate estimates under metric relative error bounds, finite-length and single-point convergence under a Kurdyka-Lojasiewicz hypothesis, and local exponential convergence near nondegenerate local minima. In particular, a metric DC-PL inequality yields exponential decay in function value and, under an additional quadratic-growth condition, in distance to the minimizing set. We also show that the convex component of a DC decomposition determines the geometry of the flow: different decompositions of the same objective lead to different continuous dynamics, and the local rate is governed by the induced metric-curvature pair. This yields a geometric criterion for assessing DC decomposition quality. Finally, we discuss how the dual-coordinate viewpoint connects DCA with Bregman proximal geometry and suggests a differential-inclusion route toward nonsmooth DC dynamics.
\end{abstract}

\noindent\textbf{Keywords.}
DCA, damped DCA, gradient flow, Hessian-Riemannian metric, Kurdyka-Lojasiewicz inequality, DC decomposition.

\medskip
\noindent\textbf{MSC 2020.}
90C26, 34A12, 90C30, 49M37, 49M05.

\section{Introduction}
\label{sec:intro}

Difference-of-convex (DC) programming studies objectives of the form
\begin{equation}
\label{eq:dc-problem}
\min_{x\in\R^n} f(x) := g(x)-h(x),
\end{equation}
where \(g\) and \(h\) are proper convex functions. The difference-of-convex algorithm (DCA), introduced and developed systematically by Pham Dinh Tao and Le Thi Hoai An, is one of the central algorithmic paradigms for such problems \cite{PhamLeThi1997,PhamLeThi1998,LeThiPham2018,LeThiDinhPham2018}. In its general, possibly nonsmooth form, DCA generates a sequence \(\{x_k\}\) by selecting
\(
y_k \in \partial h(x_k)
\)
and then solving the convex subproblem
\begin{equation}
\label{eq:basic-dca-general}
x_{k+1}\in \argmin_{x\in\R^n}\bigl\{g(x)-\ip{y_k}{x}\bigr\},
\end{equation}
or equivalently
\begin{equation}
\label{eq:basic-dca-subdiff}
y_k \in \partial h(x_k)\cap \partial g(x_{k+1}).
\end{equation}
Thus DCA linearizes the concave part \(-h\) at the current iterate and minimizes the resulting convex surrogate. In the smooth setting considered in most of this paper, \(\partial g=\{\grad g\}\) and \(\partial h=\{\grad h\}\), so the basic DCA step reduces to
\begin{equation}
\label{eq:basic-dca}
\grad g(x_{k+1}) = \grad h(x_k),
\end{equation}
which is the specialization of \eqref{eq:basic-dca-subdiff} to differentiable data.

The discrete theory of DCA is by now rich: one has monotonicity of objective values, asymptotic stationarity, rate guarantees under error-bound or KL-type assumptions, and increasingly refined Bregman interpretations \cite{LeThiPham2018,Niu2022,FaustFawziSaunderson2023,AbbaszadehpeivastiDeKlerkZamani2024}. What is much less developed is the \emph{continuous-time} viewpoint. For gradient-type and accelerated first-order methods, mirror descent, continuous Newton methods, and manifold optimization, the associated flows have clarified both geometry and asymptotics; see, for example, \cite{SuBoydCandes2016,KricheneBayenBartlett2015,Deuflhard2011,AbsilMahonySepulchre2008}. This raises the question of whether DCA admits a comparable continuous-time model.

The starting point of the analysis is that the appropriate state variable is not \(x\) but the \emph{dual coordinate}
\(
y = \grad g(x).
\)
In that variable, \eqref{eq:basic-dca} becomes the fixed-point iteration
\(
y_{k+1} = \grad h((\grad g)^{-1}(y_k)).
\)
To pass from the unit-step iteration to a small-step limit, we also introduce a \emph{damped (under-relaxed) DCA variant}. For each \(\eta\in(0,1]\), we denote by \(\{x_k^\eta\}\) the sequence generated by the damped update
\[
\grad g(x_{k+1}^{\eta}) = (1-\eta)\grad g(x_k^{\eta}) + \eta \grad h(x_k^{\eta}),
\qquad \eta\in(0,1].
\]
For fixed \(\eta\in(0,1)\), this damped DCA scheme is also of independent algorithmic interest: it admits a Bregman-regularized interpretation together with descent, asymptotic stationarity, KL convergence under boundedness assumptions, and a quantitative rate analysis under metric DC-PL conditions. The resulting picture already shows that there is no universal optimal relaxation parameter: the strongest provable global contraction guarantee within this analysis is attained by the half-relaxed choice, whereas the asymptotically fastest local linearized behavior near a nondegenerate minimum is obtained by the full-step scheme. Hence classical DCA is exactly the unit-step explicit Euler discretization of a nonlinear ODE in \(y\), while the damped DCA variant introduced here is its explicit Euler scheme with stepsize \(\eta\). Pulling this ODE back to primal coordinates produces
\begin{equation}
\label{eq:primal-flow-intro}
\Hess g(x)\dot x = \grad h(x)-\grad g(x) = -\grad f(x).
\end{equation}
This flow is not the Euclidean gradient flow of \(f\); it is the gradient flow of \(f\) in the state-dependent metric induced by the convex part \(g\). In other words, DCA is a metric method.

This reformulation has several consequences. First, \eqref{eq:primal-flow-intro} admits the exact energy identity
\[
\frac{\dd}{\dd t}f(x(t)) = -\dot x(t)^\top \Hess g(x(t))\dot x(t),
\]
so \(f\) is a Lyapunov function even though it may be nonconvex. This is the continuous-time counterpart of the monotonic decrease property that underlies the classical convergence analysis of DCA. Second, every bounded trajectory approaches the critical set, so omega-limit points are critical, in direct analogy with the asymptotic stationarity and criticality statements known for discrete DCA iterates. Third, once the gradient is bounded from below in the inverse-Hessian metric by a relative error bound, the same energy identity yields explicit global decay rates, with a clean exponential law under a metric DC-PL condition. On invariant regions where the Hessian metric is uniformly comparable to the Euclidean one, these metric rate conditions reduce to the usual global error-bound and Polyak-Lojasiewicz assumptions up to explicit distortion constants, which makes the effect of the decomposition on quantitative rates explicit. Fourth, once the trajectory is bounded, the usual dynamical-systems machinery based on Kurdyka-Lojasiewicz inequalities becomes available and yields finite length and convergence to a single critical point, mirroring the KL-based convergence paradigm developed for DCA. Here the KL property is the standard desingularizing condition that relates function-value gaps to gradient size near the limiting critical set. Accordingly, the flow provides a dynamical-systems interpretation of several standard convergence properties of DCA. Finally, the metric explicitly depends on the chosen decomposition \(f=g-h\), so two different DC splittings of the same objective lead to different flows. This sensitivity is a structural feature of DCA rather than an artifact of the analysis.

This dependence on the decomposition is also important on the algorithmic side. In practice, the quality of a DC decomposition strongly affects the behavior of DCA, yet rigorous criteria for comparing decompositions remain limited. The continuous-time framework developed here shows that decomposition quality is, at least in part, a question of metric quality: the choice of \(g\) determines the Hessian metric, this metric governs both dissipation and local rates, and different decompositions can therefore be compared through the geometry they induce.

\paragraph{Contributions.}
The main contributions of the paper are as follows.
\begin{enumerate}
\item We introduce a damped dual-coordinate variant of DCA, identify it as a Bregman-regularized DCA step with monotone descent, asymptotic stationarity, KL convergence under standard boundedness assumptions, and a metric-DC-PL linear rate bound, use it to derive a rigorous continuous limit, and show that classical DCA appears as the full-step Euler discretization of the same dual ODE. The same analysis also reveals that the relaxation-parameter choice has an intrinsic global-versus-local tradeoff: the half-relaxed scheme gives the strongest provable global guarantee furnished by the metric-DC-PL estimate, whereas the full-step scheme is locally fastest at the linearized level near nondegenerate minima.
\item We prove a metric energy identity, asymptotic criticality of bounded trajectories, explicit global rate estimates under metric relative error bounds together with their comparison to classical Euclidean error-bound and Polyak-Lojasiewicz conditions, KL-based finite-length convergence, and local exponential stability near nondegenerate local minima, thereby recovering in continuous time the main qualitative and quantitative convergence mechanisms familiar from DCA.
\item We isolate the role of the convex part \(g\) as a \emph{geometry generator}: it determines the Hessian-Riemannian metric, the instantaneous dissipation rate, the continuous path, and the local linearized rate through the spectrum of \( \Hess g(x^\star)^{-1}\Hess f(x^\star)\).
\item We show by theory and example that different DC decompositions of the same objective induce different continuous dynamics, and we use this to formulate a geometric criterion for DC decomposition quality in terms of metric alignment and curvature matching.
\item We explain how the dual-coordinate formulation interfaces naturally with Bregman proximal geometry and outline a route toward nonsmooth DC differential inclusions.
\end{enumerate}

The paper is organized as follows. Section~\ref{sec:dual} introduces the damped DCA variant, establishes its basic discrete properties, and derives the flow. Section~\ref{sec:analysis} studies Lyapunov decrease, asymptotic criticality, KL convergence, global rate estimates, and local exponential stability. Section~\ref{sec:geometry} discusses decomposition sensitivity, metric design, and the resulting interpretation of decomposition quality. Section~\ref{sec:nonsmooth} places the flow in a Bregman-geometric framework and discusses nonsmooth extensions.

\section{Dual-Coordinate Derivation of the Continuous DCA Flow}
\label{sec:dual}

Throughout the paper, except in the final outlook on nonsmooth extensions, we impose the standing assumptions
\begin{equation}
\label{eq:standing}
g,h \in C^2(\R^n), \qquad g \text{ is } \mu\text{-strongly convex for some } \mu>0.
\end{equation}
Write
\(
G(x) := \Hess g(x).
\)
Then \(G(x)\succeq \mu I\) for all \(x\in\R^n\). Since \(\grad g\) is strongly monotone and coercive, it is a global \(C^1\)-diffeomorphism from \(\R^n\) onto \(\R^n\).

\subsection{DCA as an Euler scheme in dual coordinates}

The discrete DCA update \eqref{eq:basic-dca} can be rewritten exactly after the change of variable \(y_k:=\grad g(x_k)\). For the continuous-time limit, we introduce in parallel a damped DCA family obtained by damping the dual variable with a parameter \(\eta\in(0,1]\). In primal coordinates, this means that for each \(\eta\in(0,1]\) the next iterate is defined by
\begin{equation}
\label{eq:relaxed-dca-primal}
\grad g(x_{k+1}^{\eta}) = (1-\eta)\grad g(x_k^{\eta}) + \eta \grad h(x_k^{\eta}).
\end{equation}
When \(\eta=1\), this reduces to the original DCA step. We record the resulting iteration in algorithmic form for later reference.

\begin{algorithm}[t]
\caption{Damped DCA}
\label{alg:damped-dca}
\begin{algorithmic}[1]
\Require Relaxation parameter \(\eta\in(0,1]\) and initial point \(x_0\in\R^n\)
\For{\(k=0,1,2,\dots\)}
\State Compute \(z_k \gets \grad h(x_k)\).
\State Compute \(x_{k+1}\) by
\Statex \hspace{\algorithmicindent}\(x_{k+1}\in\argmin_{x\in\R^n}
\set{
g(x)-\bigl\langle (1-\eta)\grad g(x_k)+\eta z_k,\,x\bigr\rangle
}.\)
\EndFor
\end{algorithmic}
\end{algorithm}

\begin{lemma}
\label{lem:dual-form}
Let \(T:\R^n\to\R^n\) be defined by
\begin{equation}
\label{eq:T-map}
T(y) := \grad h((\grad g)^{-1}(y)).
\end{equation}
Then the DCA iteration \eqref{eq:basic-dca} is equivalent to
\begin{equation}
\label{eq:dca-dual}
y_{k+1} = T(y_k).
\end{equation}
More generally, the damped DCA step \eqref{eq:relaxed-dca-primal} is equivalent in dual coordinates to
\begin{equation}
\label{eq:relaxed-dca-dual}
y_{k+1}^{\eta} = y_k^{\eta} + \eta\bigl(T(y_k^{\eta})-y_k^{\eta}\bigr).
\end{equation}
\end{lemma}

\begin{proof}
Set \(y_k=\grad g(x_k)\). Since \(\grad g\) is invertible, \(x_k=((\grad g)^{-1})(y_k)\). The DCA condition
\(
\grad g(x_{k+1})=\grad h(x_k)
\)
then becomes
\[
y_{k+1} = \grad h((\grad g)^{-1}(y_k)) = T(y_k).
\]
The relaxed formula follows by the same substitution:
\[
y_{k+1}^{\eta}
= (1-\eta)y_k^{\eta} + \eta \grad h((\grad g)^{-1}(y_k^{\eta}))
= y_k^{\eta} + \eta\bigl(T(y_k^{\eta})-y_k^{\eta}\bigr).
\]
\end{proof}

Lemma~\ref{lem:dual-form} immediately reveals an important fact: classical DCA is the explicit Euler method with stepsize \(1\) applied to the autonomous system
\begin{equation}
\label{eq:dual-ode}
\dot y = T(y)-y = \grad h((\grad g)^{-1}(y)) - y.
\end{equation}
Moreover, \eqref{eq:relaxed-dca-dual} can be rewritten as
\[
\frac{y_{k+1}^{\eta}-y_k^{\eta}}{\eta} = T(y_k^{\eta})-y_k^{\eta},
\]
which is exactly the forward (explicit) Euler discretization of \eqref{eq:dual-ode} with timestep \(\eta\). In particular, the original DCA corresponds to the full-step choice \(\eta=1\), while the damped family provides the small-step regime needed for a continuous-time limit.

\subsection{The continuous-time limit}

For fixed \(\eta>0\), let \(\{y_k^{\eta}\}_{k\ge 0}\) be generated by \eqref{eq:relaxed-dca-dual}, and define the piecewise affine interpolation
\[
y^{\eta}(t)
= y_k^{\eta} + \frac{t-k\eta}{\eta}\bigl(y_{k+1}^{\eta}-y_k^{\eta}\bigr),
\qquad t\in[k\eta,(k+1)\eta].
\]
Set \(x^{\eta}(t):=((\grad g)^{-1})(y^{\eta}(t))\).

\begin{theorem}
\label{thm:continuous-limit}
Let \(S>0\). Suppose the solution \(y(\cdot)\) of \eqref{eq:dual-ode} with \(y(0)=\grad g(x_0)\) stays in a compact subset of \(\R^n\) on \([0,S]\). Then, as \(\eta\downarrow 0\), the interpolants \(y^{\eta}\) converge uniformly on \([0,S]\) to \(y\), and the primal interpolants \(x^{\eta}\) converge uniformly to the unique solution \(x\) of
\begin{equation}
\label{eq:continuous-dca}
G(x(t))\dot x(t) = \grad h(x(t))-\grad g(x(t)) = -\grad f(x(t)),
\qquad x(0)=x_0.
\end{equation}
\end{theorem}

\begin{proof}
Because \(g,h\in C^2\) and \((\grad g)^{-1}\) is \(C^1\), the vector field \(y\mapsto T(y)-y\) is locally Lipschitz. Therefore the ODE \eqref{eq:dual-ode} has a unique solution on \([0,S]\), and the explicit Euler scheme \eqref{eq:relaxed-dca-dual} converges uniformly to it on every compact time interval on which the exact solution remains bounded. This is the classical convergence theorem for Euler discretization of a locally Lipschitz ODE.

Now define \(x(t):=((\grad g)^{-1})(y(t))\). Since \(y(t)=\grad g(x(t))\), differentiation yields
\[
\dot y(t) = G(x(t))\dot x(t).
\]
Combining this identity with \eqref{eq:dual-ode}, we obtain
\[
G(x(t))\dot x(t)
= \grad h(x(t))-\grad g(x(t))
= -\grad f(x(t)),
\]
which is \eqref{eq:continuous-dca}. Uniform convergence of \(x^{\eta}\) follows from uniform convergence of \(y^{\eta}\) and continuity of \((\grad g)^{-1}\).
\end{proof}

Equation \eqref{eq:continuous-dca} will be the main object of study in the rest of the paper. It arises as the continuous-time model associated with DCA through the vanishing-step limit of the damped variant introduced above. The original DCA corresponds to the full-step Euler scheme in the dual geometry generated by \(g\).

\subsection{Discrete properties of the damped scheme}

The damped family is not merely a device for passing to continuous time. For each fixed relaxation parameter \(\eta\in(0,1)\), it defines a Bregman-regularized DCA step with its own descent mechanism and convergence properties.

For \(z,x\in\R^n\), define the Bregman divergence generated by \(g\) by
\begin{equation}
\label{eq:bregman-def}
D_g(z,x):=g(z)-g(x)-\ip{\grad g(x)}{z-x}.
\end{equation}

\begin{proposition}
\label{prop:relaxed-descent}
Fix \(\eta\in(0,1)\), and let \(\{x_k\}\) be generated by the damped DCA step \eqref{eq:relaxed-dca-primal}. Then \(x_{k+1}\) solves
\begin{equation}
\label{eq:relaxed-proximal}
x_{k+1}\in \argmin_{x\in\R^n}
\Bigl\{
\eta\bigl[g(x)-h(x_k)-\ip{\grad h(x_k)}{x-x_k}\bigr]
+(1-\eta)D_g(x,x_k)
\Bigr\}.
\end{equation}
Consequently,
\begin{equation}
\label{eq:relaxed-descent}
f(x_{k+1})+\frac{1-\eta}{\eta}D_g(x_{k+1},x_k)\le f(x_k).
\end{equation}
If \(g\) is \(\mu\)-strongly convex, then
\begin{equation}
\label{eq:relaxed-strong-descent}
f(x_k)-f(x_{k+1})
\ge \frac{(1-\eta)\mu}{2\eta}\norm{x_{k+1}-x_k}^2.
\end{equation}
In particular, if \(f\) is bounded below along \(\{x_k\}\), then \(\{f(x_k)\}\) is decreasing and
\[
\sum_{k=0}^{\infty} D_g(x_{k+1},x_k)<\infty,
\qquad
\sum_{k=0}^{\infty}\norm{x_{k+1}-x_k}^2<\infty.
\]
\end{proposition}

\begin{proof}
For fixed \(k\), consider the function
\[
\Psi_k(x):=
\eta\bigl[g(x)-h(x_k)-\ip{\grad h(x_k)}{x-x_k}\bigr]
+(1-\eta)D_g(x,x_k).
\]
Using \eqref{eq:bregman-def}, one can rewrite \(\Psi_k\) up to an additive constant independent of \(x\) as
\[
g(x)-\ip{(1-\eta)\grad g(x_k)+\eta\grad h(x_k)}{x}.
\]
Therefore the optimality condition for minimizing \(\Psi_k\) is exactly
\[
\grad g(x_{k+1})=(1-\eta)\grad g(x_k)+\eta\grad h(x_k),
\]
which is \eqref{eq:relaxed-dca-primal}. This proves \eqref{eq:relaxed-proximal}.

Since \(h\) is convex,
\[
h(x)\ge h(x_k)+\ip{\grad h(x_k)}{x-x_k},
\]
and therefore
\[
g(x)-h(x_k)-\ip{\grad h(x_k)}{x-x_k}\ge f(x)
\qquad \forall x\in\R^n.
\]
Using \(\Psi_k(x_{k+1})\le \Psi_k(x_k)=\eta f(x_k)\), we obtain
\[
\eta f(x_{k+1})+(1-\eta)D_g(x_{k+1},x_k)
\le \Psi_k(x_{k+1})
\le \eta f(x_k),
\]
which is \eqref{eq:relaxed-descent}. Strong convexity of \(g\) gives
\[
D_g(x_{k+1},x_k)\ge \frac{\mu}{2}\norm{x_{k+1}-x_k}^2,
\]
and \eqref{eq:relaxed-strong-descent} follows. Summing \eqref{eq:relaxed-descent} and \eqref{eq:relaxed-strong-descent} over \(k\) gives the final claim.
\end{proof}

\begin{proposition}
\label{prop:relaxed-criticality}
Fix \(\eta\in(0,1)\), and let \(\{x_k\}\) be generated by \eqref{eq:relaxed-dca-primal}. Suppose that \(\{x_k\}\) is bounded. Then
\[
\norm{x_{k+1}-x_k}\to 0,
\qquad
\norm{\grad f(x_k)}\to 0.
\]
In particular, every cluster point of \(\{x_k\}\) is a critical point of \(f\).
\end{proposition}

\begin{proof}
Since \(\{x_k\}\) is bounded and \(f\) is continuous, \(f\) is bounded below along the sequence. Proposition~\ref{prop:relaxed-descent} therefore yields
\[
\sum_{k=0}^{\infty}\norm{x_{k+1}-x_k}^2<\infty,
\]
hence \(\norm{x_{k+1}-x_k}\to 0\).

Let \(K\subset\R^n\) be a compact set containing the whole sequence. Since \(g\in C^2\), the gradient \(\grad g\) is Lipschitz on \(K\); let \(L_g>0\) be a corresponding Lipschitz constant. From \eqref{eq:relaxed-dca-primal},
\[
\grad g(x_{k+1})-\grad g(x_k)
=\eta\bigl(\grad h(x_k)-\grad g(x_k)\bigr)
=-\eta\grad f(x_k).
\]
Therefore
\[
\eta\norm{\grad f(x_k)}
=\norm{\grad g(x_{k+1})-\grad g(x_k)}
\le L_g\norm{x_{k+1}-x_k}\to 0.
\]
This proves \(\grad f(x_k)\to 0\). Any cluster point \(x^\star\) then satisfies \(\grad f(x^\star)=0\) by continuity of \(\grad f\).
\end{proof}

\section{Lyapunov Analysis and Long-Time Behavior}
\label{sec:analysis}

\subsection{Energy identity and metric gradient structure}

The ODE \eqref{eq:continuous-dca} is naturally expressed in the metric generated by \(G(x)\).

\begin{proposition}
\label{prop:energy}
Let \(x(\cdot)\) solve \eqref{eq:continuous-dca} on an interval \([0,\tau)\). Then
\begin{equation}
\label{eq:energy-identity}
\frac{\dd}{\dd t}f(x(t))
= -\grad f(x(t))^\top G(x(t))^{-1}\grad f(x(t))
= -\dot x(t)^\top G(x(t))\dot x(t)
\le 0.
\end{equation}
Hence \(f\) is a Lyapunov function for the flow. Equivalently, \eqref{eq:continuous-dca} is the gradient flow of \(f\) in the Riemannian metric
\[
\ip{u}{v}_{x} := u^\top G(x)v.
\]
For any symmetric positive-definite matrix \(A\), we write \(\norm{v}_{A}^2:=v^\top Av\).
\end{proposition}

\begin{proof}
Using \eqref{eq:continuous-dca},
\[
\frac{\dd}{\dd t}f(x(t))
= \ip{\grad f(x(t))}{\dot x(t)}
= -\grad f(x(t))^\top G(x(t))^{-1}\grad f(x(t)).
\]
Since \(G(x)\succ 0\), the right-hand side is nonpositive. The identity
\(
\grad f = -G\dot x
\)
then gives the second equality in \eqref{eq:energy-identity}. The Riemannian gradient associated with the metric \(\ip{\cdot}{\cdot}_x\) is \(G(x)^{-1}\grad f(x)\), so the flow is precisely \(\dot x=-\mathrm{grad}_{G}f(x)\).
\end{proof}

Proposition~\ref{prop:energy} yields a useful metric interpretation:
\begin{equation}
\label{eq:metric-speed}
\frac{\dd}{\dd t}f(x(t)) = -\norm{\dot x(t)}_{G(x(t))}^2.
\end{equation}
Thus the instantaneous decay of the objective equals the squared speed of the trajectory in the Hessian metric of \(g\).

\subsection{Asymptotic criticality of bounded trajectories}

The energy identity implies integrability of the metric speed, and with a mild additional regularity assumption one recovers asymptotic stationarity.

\begin{proposition}
\label{prop:criticality}
Assume \(g,h\in C^3(\R^n)\), and let \(x(\cdot)\) be a bounded global solution of \eqref{eq:continuous-dca}. If \(f\) is bounded below on the trajectory, then
\begin{equation}
\label{eq:l2-speed}
\int_0^{+\infty} \norm{\dot x(t)}_{G(x(t))}^2 \dd t < +\infty,
\end{equation}
and
\begin{equation}
\label{eq:grad-goes-zero}
\lim_{t\to+\infty}\norm{\grad f(x(t))}=0.
\end{equation}
Consequently, every cluster point of \(x(t)\) is a critical point of \(f\).
\end{proposition}

\begin{proof}
By Proposition~\ref{prop:energy}, the function \(t\mapsto f(x(t))\) is decreasing and bounded below, hence it converges to some \(f_\infty\in\R\). Integrating \eqref{eq:metric-speed} gives
\[
\int_0^{+\infty} \norm{\dot x(t)}_{G(x(t))}^2 \dd t
= f(x(0))-f_\infty < +\infty.
\]
This proves \eqref{eq:l2-speed}.

Because the trajectory is bounded and \(g,h\in C^3\), the sets
\(
\{x(t):t\ge 0\}
\)
and
\(
\{\dot x(t):t\ge 0\}
\)
are bounded, and the map
\[
w(t):=\grad f(x(t))^\top G(x(t))^{-1}\grad f(x(t))
= \norm{\dot x(t)}_{G(x(t))}^2
\]
has bounded derivative. Hence \(w\) is uniformly continuous on \([0,+\infty)\). Since \(w\ge 0\) and \(w\in L^1(0,+\infty)\), Barbalat's lemma implies \(w(t)\to 0\).

Now boundedness of the trajectory and continuity of \(G\) imply the existence of \(M>0\) such that \(G(x(t))\preceq MI\) for all \(t\ge 0\). Therefore
\[
w(t)
= \grad f(x(t))^\top G(x(t))^{-1}\grad f(x(t))
\ge \frac{1}{M}\norm{\grad f(x(t))}^2.
\]
Thus \eqref{eq:grad-goes-zero} holds. Any cluster point \(x^\star\) must then satisfy \(\grad f(x^\star)=0\).
\end{proof}

\subsection{KL convergence and finite-length trajectories}

The energy identity becomes especially powerful when \(f\) satisfies a Kurdyka-Lojasiewicz (KL) inequality on the omega-limit set, which is the natural setting for tame, subanalytic, and real-analytic objectives \cite{Kurdyka1998,BolteDaniilidisLewisShiota2007,AttouchBolteSvaiter2013}. Recall that the KL property is a local desingularization principle: near a critical point, the gap \(f(x)-f^\star\) controls the size of \(\grad f(x)\) after composition with a suitable increasing concave function. This is precisely the mechanism that converts mere asymptotic criticality into finite-length convergence.

\begin{definition}
\label{def:kl}
Let \(f:\R^n\to\R\) be differentiable, and let \(x^\star\in\R^n\) satisfy \(\grad f(x^\star)=0\). We say that \(f\) has the \emph{Kurdyka-Lojasiewicz (KL) property} at \(x^\star\) if there exist \(\varepsilon>0\), a neighborhood \(U\) of \(x^\star\), and a concave function
\[
\varphi\in C^1\bigl((0,\varepsilon)\bigr)\cap C\bigl([0,\varepsilon)\bigr),
\qquad
\varphi(0)=0,\quad \varphi'>0,
\]
such that
\begin{equation}
\label{eq:kl-def}
\varphi'(f(x)-f(x^\star))\norm{\grad f(x)} \ge 1
\end{equation}
for every \(x\in U\) satisfying \(f(x^\star)< f(x) < f(x^\star)+\varepsilon\).
\end{definition}

\begin{theorem}
\label{thm:kl}
Let \(x(\cdot)\) be a bounded global solution of \eqref{eq:continuous-dca}. Assume \(g,h\in C^3(\R^n)\), \(f\) is bounded below, and \(f\) satisfies the KL property at every point of the omega-limit set
\[
\omega(x_0):=\bigcap_{s\ge 0}\overline{\set{x(t):t\ge s}}.
\]
Then the trajectory \(x(\cdot)\) has finite Euclidean length:
\begin{equation}
\label{eq:finite-length}
\int_0^{+\infty}\norm{\dot x(t)}\dd t < +\infty.
\end{equation}
In particular, \(x(t)\) converges to a single critical point \(x^\star\) of \(f\).
\end{theorem}

\begin{proof}
By Proposition~\ref{prop:energy}, \(f(x(t))\downarrow f_\infty\). Since the trajectory is bounded, its omega-limit set is nonempty, compact, and connected. Standard uniformization of the KL inequality on compact sets implies that there exist \(\varepsilon>0\), a neighborhood \(U\) of \(\omega(x_0)\), and a concave function
\[
\varphi\in C^1\bigl((0,\varepsilon)\bigr)\cap C\bigl([0,\varepsilon)\bigr),
\qquad
\varphi(0)=0,\quad \varphi'>0,
\]
such that
\begin{equation}
\label{eq:uniform-kl}
\varphi'(f(x)-f_\infty)\norm{\grad f(x)} \ge 1
\end{equation}
whenever \(x\in U\) and \(f_\infty < f(x) < f_\infty+\varepsilon\).

For \(t\) large enough, \(x(t)\in U\) and \(f_\infty < f(x(t)) < f_\infty+\varepsilon\). On the compact tail of the trajectory there exist constants \(0<\mu\le M\) such that
\(
\mu I \preceq G(x(t)) \preceq MI
\)
for all large \(t\). Using \eqref{eq:uniform-kl}, Proposition~\ref{prop:energy}, and the identity
\(
\norm{\dot x}_{G}^2 = -\frac{\dd}{\dd t}f(x(t)),
\)
we obtain for all large \(t\),
\begin{align*}
-\frac{\dd}{\dd t}\varphi(f(x(t))-f_\infty)
&= \varphi'(f(x(t))-f_\infty)\norm{\dot x(t)}_{G(x(t))}^2 \\
&\ge \frac{\norm{\dot x(t)}_{G(x(t))}^2}{\norm{\grad f(x(t))}}.
\end{align*}
Because \(\grad f(x) = -G(x)\dot x\),
\[
\norm{\grad f(x(t))}
\le \norm{G(x(t))}^{1/2}\norm{\dot x(t)}_{G(x(t))}
\le \sqrt{M}\,\norm{\dot x(t)}_{G(x(t))}.
\]
Therefore
\[
-\frac{\dd}{\dd t}\varphi(f(x(t))-f_\infty)
\ge \frac{1}{\sqrt{M}}\norm{\dot x(t)}_{G(x(t))}
\ge \sqrt{\frac{\mu}{M}}\,\norm{\dot x(t)}.
\]
Integrating over \([t_0,+\infty)\) for \(t_0\) sufficiently large yields
\[
\int_{t_0}^{+\infty}\norm{\dot x(t)}\dd t
\le \sqrt{\frac{M}{\mu}}\,
\varphi\bigl(f(x(t_0))-f_\infty\bigr)
<+\infty.
\]
This proves \eqref{eq:finite-length}.

Finite length implies that \(x(\cdot)\) is a Cauchy curve in \(\R^n\), hence \(x(t)\to x^\star\) for some \(x^\star\in\R^n\). Proposition~\ref{prop:criticality} then gives \(\grad f(x^\star)=0\).
\end{proof}

Theorem~\ref{thm:kl} shows that the Hessian metric induced by \(g\) is fully compatible with the standard tame-convergence paradigm: once the metric is uniformly equivalent to the Euclidean one on bounded sets, the usual KL machinery applies with only minor modifications.

The same KL mechanism applies to the fixed-step damped scheme as well.

\begin{theorem}
\label{thm:relaxed-kl}
Fix \(\eta\in(0,1)\), and let \(\{x_k\}\) be generated by \eqref{eq:relaxed-dca-primal}. Assume that \(\{x_k\}\) is bounded and that \(f\) satisfies the KL property at every point of the cluster set of \(\{x_k\}\). Then
\[
\sum_{k=0}^{\infty}\norm{x_{k+1}-x_k}<\infty.
\]
In particular, \(\{x_k\}\) converges to a single critical point of \(f\).
\end{theorem}

\begin{proof}
By Proposition~\ref{prop:relaxed-descent}, the sequence \(\{f(x_k)\}\) is decreasing and bounded below, hence convergent, and there exists
\[
a:=\frac{(1-\eta)\mu}{2\eta}>0
\]
such that
\begin{equation}
\label{eq:relaxed-sufficient-decrease}
f(x_k)-f(x_{k+1})\ge a\norm{x_{k+1}-x_k}^2
\qquad \forall k\ge 0.
\end{equation}
By Proposition~\ref{prop:relaxed-criticality}, \(\norm{x_{k+1}-x_k}\to 0\) and every cluster point of \(\{x_k\}\) is critical.

Let \(K\subset\R^n\) be a compact set containing the whole sequence. Since \(g,h\in C^2\), the gradients \(\grad g\) and \(\grad h\) are Lipschitz on \(K\); denote the corresponding Lipschitz constants by \(L_g\) and \(L_h\). From \eqref{eq:relaxed-dca-primal},
\begin{align*}
\grad f(x_{k+1})
&= \grad g(x_{k+1})-\grad h(x_{k+1}) \\
&= (1-\eta)\bigl(\grad g(x_k)-\grad h(x_k)\bigr)
+\grad h(x_k)-\grad h(x_{k+1}),
\end{align*}
hence
\begin{align*}
\norm{\grad f(x_{k+1})}
&\le \frac{1-\eta}{\eta}\norm{\grad g(x_{k+1})-\grad g(x_k)}
+\norm{\grad h(x_{k+1})-\grad h(x_k)} \\
&\le \Bigl(\frac{1-\eta}{\eta}L_g+L_h\Bigr)\norm{x_{k+1}-x_k}.
\end{align*}
Thus the sequence satisfies a standard relative-error condition of the form
\begin{equation}
\label{eq:relaxed-relative-error}
\norm{\grad f(x_{k+1})}\le b\norm{x_{k+1}-x_k}
\qquad \forall k\ge 0
\end{equation}
for some constant \(b>0\).

The sufficient decrease estimate \eqref{eq:relaxed-sufficient-decrease}, the relative-error bound \eqref{eq:relaxed-relative-error}, and the KL property on the cluster set place \(\{x_k\}\) exactly in the abstract framework of KL convergence theorems for descent methods; see, for example, \cite{AttouchBolteSvaiter2013,Niu2022}. Therefore the sequence has finite length and converges to a single cluster point \(x^\star\). Proposition~\ref{prop:relaxed-criticality} yields \(\grad f(x^\star)=0\).
\end{proof}

\subsection{Global rates under metric relative error bounds}

The KL theorem is qualitative unless the desingularizing function is quantified. For explicit global rates, a natural specialization is a relative error bound measured in the inverse-Hessian metric generated by \(g\). This is the continuous-time counterpart of rate assumptions used in discrete DCA analyses and in the Polyak-Lojasiewicz theory of gradient methods \cite{Niu2022,AbbaszadehpeivastiDeKlerkZamani2024,Polyak1963,KarimiNutiniSchmidt2016}.

\begin{definition}
\label{def:metric-reb}
Let \(\Omega\subset\R^n\) be positively invariant for \eqref{eq:continuous-dca}, and set
\[
f_\star:=\inf_{x\in\Omega} f(x).
\]
We say that \(f\) satisfies a \emph{metric relative error bound of exponent \(\theta\in[1/2,1)\)} on \(\Omega\) if there exists \(c>0\) such that
\begin{equation}
\label{eq:metric-reb}
\norm{\grad f(x)}_{G(x)^{-1}} \ge c\bigl(f(x)-f_\star\bigr)^\theta
\end{equation}
for every \(x\in\Omega\) with \(f(x)>f_\star\). In the special case \(\theta=\frac12\), equivalently
\begin{equation}
\label{eq:metric-pl}
\norm{\grad f(x)}_{G(x)^{-1}}^2 \ge 2\mu\bigl(f(x)-f_\star\bigr)
\qquad (2\mu=c^2),
\end{equation}
we say that \(f\) satisfies a \emph{metric DC-PL inequality} on \(\Omega\).
\end{definition}

This is a power-type quantified KL inequality written in the metric induced by \(g\), but imposed globally on the invariant set \(\Omega\) rather than only near a limiting critical point.

\begin{proposition}
\label{prop:metric-vs-euclidean-rates}
Assume that \(\theta\in[1/2,1)\) and that there exist constants \(0<m\le M\) such that
\[
mI\preceq G(x)\preceq MI
\qquad \forall x\in\Omega.
\]
Then:
\begin{enumerate}
\item if there exists \(\kappa>0\) such that
\[
\norm{\grad f(x)}\ge \kappa \bigl(f(x)-f_\star\bigr)^\theta
\qquad \forall x\in\Omega \text{ with } f(x)>f_\star,
\]
then \(f\) satisfies the metric relative error bound \eqref{eq:metric-reb} on \(\Omega\) with \(c=\kappa/\sqrt{M}\);
\item if \(f\) satisfies the metric relative error bound \eqref{eq:metric-reb} on \(\Omega\) with constant \(c\), then
\[
\norm{\grad f(x)}\ge \sqrt{m}\,c \bigl(f(x)-f_\star\bigr)^\theta
\qquad \forall x\in\Omega \text{ with } f(x)>f_\star.
\]
\end{enumerate}
In particular, the standard Polyak-Lojasiewicz inequality
\[
\norm{\grad f(x)}^2\ge 2\mu_{\mathrm{E}}\bigl(f(x)-f_\star\bigr)
\]
implies the metric DC-PL inequality \eqref{eq:metric-pl} with \(\mu=\mu_{\mathrm{E}}/M\), whereas \eqref{eq:metric-pl} implies the standard Polyak-Lojasiewicz inequality with \(\mu_{\mathrm{E}}=m\mu\).
\end{proposition}

\begin{proof}
The bounds on \(G\) imply
\[
\frac{1}{M}I\preceq G(x)^{-1}\preceq \frac{1}{m}I
\qquad \forall x\in\Omega.
\]
Hence, for every \(v\in\R^n\),
\[
\frac{1}{M}\norm{v}^2\le \norm{v}_{G(x)^{-1}}^2\le \frac{1}{m}\norm{v}^2.
\]
The first claim follows from
\[
\norm{\grad f(x)}_{G(x)^{-1}}
\ge \frac{1}{\sqrt{M}}\norm{\grad f(x)}
\ge \frac{\kappa}{\sqrt{M}}\bigl(f(x)-f_\star\bigr)^\theta.
\]
The second claim follows from
\[
\norm{\grad f(x)}
\ge \sqrt{m}\,\norm{\grad f(x)}_{G(x)^{-1}}
\ge \sqrt{m}\,c\bigl(f(x)-f_\star\bigr)^\theta.
\]
Specializing to \(\theta=\frac12\) yields the PL comparison.
\end{proof}

\begin{theorem}
\label{thm:metric-rates}
Let \(\Omega\subset\R^n\) be positively invariant for \eqref{eq:continuous-dca}, and let \(x(\cdot)\) be a global solution with \(x(0)\in\Omega\). Assume that \(f\) is bounded below on \(\Omega\) and satisfies the metric relative error bound \eqref{eq:metric-reb} on \(\Omega\) for some \(\theta\in[1/2,1)\). Then, with
\[
V(t):=f(x(t))-f_\star,
\]
the following hold:
\begin{enumerate}
\item If \(\theta=\frac12\), then
\begin{equation}
\label{eq:metric-pl-rate}
V(t)\le e^{-c^2 t}V(0).
\end{equation}
\item If \(\theta\in(\frac12,1)\), then
\begin{equation}
\label{eq:metric-poly-rate}
V(t)\le \Bigl(V(0)^{1-2\theta}+c^2(2\theta-1)t\Bigr)^{-1/(2\theta-1)}.
\end{equation}
\end{enumerate}
If, in addition, the minimizer set
\[
\mathcal X_\star:=\{x\in\Omega:f(x)=f_\star\}
\]
is nonempty and \(f\) satisfies the quadratic-growth condition
\begin{equation}
\label{eq:qg-global}
f(x)-f_\star \ge \frac{\alpha}{2}\dist(x,\mathcal X_\star)^2,
\qquad x\in\Omega,
\end{equation}
then
\begin{equation}
\label{eq:distance-from-value}
\dist(x(t),\mathcal X_\star)\le \sqrt{\frac{2}{\alpha}}\,V(t)^{1/2}.
\end{equation}
In particular, under the metric DC-PL inequality \eqref{eq:metric-pl},
\begin{equation}
\label{eq:metric-pl-dist-rate}
\dist(x(t),\mathcal X_\star)\le \sqrt{\frac{2V(0)}{\alpha}}\,e^{-c^2 t/2}.
\end{equation}
\end{theorem}

\begin{proof}
By Proposition~\ref{prop:energy},
\[
\dot V(t)=\frac{\dd}{\dd t}f(x(t))=-\norm{\grad f(x(t))}_{G(x(t))^{-1}}^2.
\]
Using \eqref{eq:metric-reb}, we obtain
\begin{equation}
\label{eq:rate-ode}
\dot V(t)\le -c^2V(t)^{2\theta}
\end{equation}
whenever \(V(t)>0\).

If \(\theta=\frac12\), then \eqref{eq:rate-ode} reduces to
\[
\dot V(t)\le -c^2V(t),
\]
and Gronwall's lemma yields \eqref{eq:metric-pl-rate}. If \(\theta\in(\frac12,1)\), then for every \(t\) such that \(V(t)>0\),
\[
\frac{\dd}{\dd t}V(t)^{1-2\theta}
=(1-2\theta)V(t)^{-2\theta}\dot V(t)
\ge c^2(2\theta-1).
\]
Integrating from \(0\) to \(t\) gives
\[
V(t)^{1-2\theta}\ge V(0)^{1-2\theta}+c^2(2\theta-1)t,
\]
which is equivalent to \eqref{eq:metric-poly-rate}.

If \eqref{eq:qg-global} holds, then
\[
\frac{\alpha}{2}\dist(x(t),\mathcal X_\star)^2\le f(x(t))-f_\star=V(t),
\]
which proves \eqref{eq:distance-from-value}. Combining \eqref{eq:distance-from-value} with \eqref{eq:metric-pl-rate} gives \eqref{eq:metric-pl-dist-rate}.
\end{proof}

Proposition~\ref{prop:metric-vs-euclidean-rates} shows that these conditions are not alien to the standard rate theory of gradient systems: on regions where \(G\) is uniformly equivalent to the Euclidean metric, they are precisely the familiar error-bound and Polyak-Lojasiewicz inequalities, with constants distorted by the extremal eigenvalues of \(G\). In particular, if \(f\) satisfies a Euclidean PL inequality with constant \(\mu_{\mathrm{E}}\) and \(G(x)\preceq MI\) on \(\Omega\), then
\[
f(x(t))-f_\star \le e^{-2\mu_{\mathrm{E}}t/M}\bigl(f(x(0))-f_\star\bigr),
\]
so excessive convexification of \(g\) degrades the global exponential rate by the factor \(M\). The case \(\theta=\frac12\) makes this especially clear. The metric DC-PL constant depends on the inverse Hessian metric induced by \(g\), and an unnecessarily large convexifier directly worsens the exponential decay constant through the factor \(M^{-1}\). Consequently, quantitative global rates, no less than local ones, are decomposition-dependent. This shows, now at the global level, how DC decomposition quality can be interpreted as metric quality.

The same metric DC-PL mechanism also yields a quantitative rate for the fixed-step damped scheme and, with it, a first principled criterion for choosing the relaxation parameter at the level of provable guarantees.

\begin{theorem}
\label{thm:damped-pl-rate}
Fix \(\eta\in(0,1)\), and let \(\{x_k\}\subset K\) be generated by \eqref{eq:relaxed-dca-primal}, where \(K\subset\R^n\) is convex and compact. Assume that \(\grad g\) is \(L_g\)-Lipschitz on \(K\), and that there exists \(\sigma>0\) such that
\begin{equation}
\label{eq:damped-metric-pl}
\norm{\grad f(x)}_{G(x)^{-1}}^2 \ge 2\sigma\bigl(f(x)-f_\star\bigr)
\qquad \forall x\in K,
\end{equation}
where \(f_\star:=\inf_{x\in K}f(x)\). Then, with
\[
q_\eta:=\max\Bigl\{0,\,1-\frac{\mu\sigma}{L_g}\eta(1-\eta)\Bigr\},
\]
one has
\begin{equation}
\label{eq:damped-linear-rate}
f(x_{k+1})-f_\star \le q_\eta\bigl(f(x_k)-f_\star\bigr)
\qquad \forall k\ge 0,
\end{equation}
and hence
\[
f(x_k)-f_\star \le q_\eta^k\bigl(f(x_0)-f_\star\bigr)
\qquad \forall k\ge 0.
\]
Among \(\eta\in(0,1)\), the contraction factor \(q_\eta\) furnished by this estimate is minimized at the half-relaxed choice \(\eta=\frac12\).
\end{theorem}

\begin{proof}
Since \(g\) is convex and \(\grad g\) is \(L_g\)-Lipschitz on the convex set \(K\), the Baillon-Haddad cocoercivity inequality yields
\[
D_g(x_{k+1},x_k)
\ge \frac{1}{2L_g}\norm{\grad g(x_{k+1})-\grad g(x_k)}^2.
\]
Using \eqref{eq:relaxed-dca-primal}, we obtain
\[
\grad g(x_{k+1})-\grad g(x_k)
=\eta\bigl(\grad h(x_k)-\grad g(x_k)\bigr)
=-\eta\grad f(x_k),
\]
and therefore
\[
D_g(x_{k+1},x_k)\ge \frac{\eta^2}{2L_g}\norm{\grad f(x_k)}^2.
\]
Because \(g\) is \(\mu\)-strongly convex, \(G(x_k)\succeq \mu I\), hence
\[
\norm{\grad f(x_k)}^2 \ge \mu \norm{\grad f(x_k)}_{G(x_k)^{-1}}^2.
\]
Combining this with \eqref{eq:damped-metric-pl} gives
\[
D_g(x_{k+1},x_k)\ge \frac{\mu\sigma}{L_g}\eta^2\bigl(f(x_k)-f_\star\bigr).
\]
Now \eqref{eq:relaxed-descent} implies
\begin{align*}
f(x_{k+1})-f_\star
&\le f(x_k)-f_\star-\frac{1-\eta}{\eta}D_g(x_{k+1},x_k) \\
&\le \Bigl(1-\frac{\mu\sigma}{L_g}\eta(1-\eta)\Bigr)\bigl(f(x_k)-f_\star\bigr),
\end{align*}
\[
\tilde q_\eta:=1-\frac{\mu\sigma}{L_g}\eta(1-\eta).
\]
If \(\tilde q_\eta\ge 0\), then the previous estimate is exactly \eqref{eq:damped-linear-rate}, since in that case \(q_\eta=\tilde q_\eta\). If \(\tilde q_\eta<0\), then the same estimate, together with \(f(x_{k+1})-f_\star\ge 0\) and \(f(x_k)-f_\star\ge 0\), forces
\[
f(x_k)-f_\star = f(x_{k+1})-f_\star = 0.
\]
Hence \eqref{eq:damped-linear-rate} still holds, now with \(q_\eta=0\). Therefore the one-step estimate is valid in all cases. Iterating it yields the linear bound. Finally, the map \(s\mapsto \max\{0,1-\frac{\mu\sigma}{L_g}s\}\) is nonincreasing for \(s\ge 0\), while \(\eta(1-\eta)\) is maximized at \(\eta=\frac12\), so \(q_\eta\) is minimized at the half-relaxed choice \(\eta=\frac12\).
\end{proof}

\subsection{Local exponential stability near nondegenerate local minima}

When a global metric relative error bound is unavailable, the metric viewpoint still yields a clear local rate statement near nondegenerate minima. Recall that an equilibrium of \eqref{eq:continuous-dca} is simply a point \(x^\star\) for which the constant trajectory \(x(t)\equiv x^\star\) solves the flow, equivalently \(\grad f(x^\star)=0\). Saying that such an equilibrium is \emph{locally exponentially stable} means that every trajectory starting sufficiently near \(x^\star\) remains near it and converges back to \(x^\star\) at an exponential rate. This is the continuous-time analogue of a local linear convergence statement for an algorithm and shows that, near a nondegenerate local minimum, the DCA flow is not only asymptotically convergent but quantitatively contracting.

\begin{theorem}
\label{thm:local-exp}
Let \(x^\star\) be a critical point of \(f\) such that
\begin{equation}
\label{eq:strict-local-min}
\Hess f(x^\star)\succ 0.
\end{equation}
Then \(x^\star\) is a locally exponentially stable equilibrium of \eqref{eq:continuous-dca}. More precisely, there exist neighborhoods \(U\subset V\) of \(x^\star\) and constants \(c_1,c_2,\lambda>0\) such that every solution with \(x(0)\in U\) remains in \(V\) for all \(t\ge 0\) and satisfies
\begin{equation}
\label{eq:exp-estimate}
\norm{x(t)-x^\star}
\le c_1 e^{-\lambda t}\norm{x(0)-x^\star},
\qquad
f(x(t))-f(x^\star)\le c_2 e^{-2\lambda t}.
\end{equation}
In fact, if \(m_f,L_f,M>0\) and a neighborhood \(V\) of \(x^\star\) are chosen so that
\[
m_f I \preceq \Hess f(x)\preceq L_f I,
\qquad
\Hess g(x)\preceq MI
\]
for all \(x\in V\), then for every sufficiently small neighborhood \(U\) of \(x^\star\) with \(\overline U\subset V\), one may take
\[
\lambda=\frac{m_f}{M},
\qquad
c_1=\sqrt{\frac{L_f}{m_f}},
\qquad
c_2=\sup_{x\in U}\bigl(f(x)-f(x^\star)\bigr).
\]
\end{theorem}

\begin{proof}
Condition \eqref{eq:strict-local-min} and continuity of \(\Hess f\) imply that, after shrinking to a sufficiently small neighborhood \(V\) of \(x^\star\), there exist constants \(m_f,L_f>0\) such that
\[
m_f I \preceq \Hess f(x)\preceq L_f I,
\qquad x\in V.
\]
Since \(\grad f(x^\star)=0\), Taylor's theorem yields the quadratic bounds
\begin{equation}
\label{eq:local-strong-convex}
f(x)-f(x^\star)\ge \frac{m_f}{2}\norm{x-x^\star}^2
\end{equation}
and
\begin{equation}
\label{eq:local-upper-quad}
f(x)-f(x^\star)\le \frac{L_f}{2}\norm{x-x^\star}^2
\end{equation}
for all \(x\in V\). Shrinking \(V\) if necessary, we may also assume the local Polyak-Lojasiewicz estimate
\begin{equation}
\label{eq:local-pl}
\norm{\grad f(x)}^2 \ge 2m_f\bigl(f(x)-f(x^\star)\bigr),
\qquad x\in V,
\end{equation}
and that \(G(x)\preceq MI\) on \(V\) for some \(M>0\).

Choose \(r>0\) so that the closed ball \(\overline{B_r(x^\star)}\) is contained in \(V\), and let
\[
\delta:=\frac{m_f}{2}r^2.
\]
By \eqref{eq:local-strong-convex}, the sublevel set
\[
\{x\in V: f(x)-f(x^\star)\le \delta\}
\]
is contained in \(B_r(x^\star)\subset V\). Now define
\[
\rho:=\sqrt{\frac{m_f}{L_f}}\,r,
\qquad
U:=B_\rho(x^\star).
\]
Then \eqref{eq:local-upper-quad} gives
\[
f(x)-f(x^\star)\le \frac{L_f}{2}\rho^2=\delta
\qquad \forall x\in U,
\]
so \(U\) is contained in the above sublevel set. Therefore, if \(x(0)\in U\), Proposition~\ref{prop:energy} implies that \(f(x(t))\le f(x(0))\le f(x^\star)+\delta\) for all \(t\ge 0\), and hence the whole trajectory remains in \(V\).

For such a trajectory, Proposition~\ref{prop:energy} and \eqref{eq:local-pl} give
\begin{align*}
\frac{\dd}{\dd t}\bigl(f(x(t))-f(x^\star)\bigr)
&= -\grad f(x(t))^\top G(x(t))^{-1}\grad f(x(t)) \\
&\le -\frac{1}{M}\norm{\grad f(x(t))}^2 \\
&\le -\frac{2m_f}{M}\bigl(f(x(t))-f(x^\star)\bigr).
\end{align*}
Gronwall's lemma yields
\[
f(x(t))-f(x^\star)
\le e^{-2m_ft/M}\bigl(f(x(0))-f(x^\star)\bigr).
\]
Using \eqref{eq:local-strong-convex} and \eqref{eq:local-upper-quad}, we obtain
\[
\frac{m_f}{2}\norm{x(t)-x^\star}^2
\le f(x(t))-f(x^\star)
\le e^{-2m_ft/M}\frac{L_f}{2}\norm{x(0)-x^\star}^2,
\]
hence
\[
\norm{x(t)-x^\star}\le \sqrt{\frac{L_f}{m_f}}\,e^{-m_ft/M}\norm{x(0)-x^\star}.
\]
This proves the first estimate in \eqref{eq:exp-estimate} with \(c_1=\sqrt{L_f/m_f}\) and \(\lambda=m_f/M\). The second estimate follows from the already established decay of \(f(x(t))-f(x^\star)\), and one may take \(c_2=\sup_{x\in U}(f(x)-f(x^\star))\).
\end{proof}

The rate constant in Theorem~\ref{thm:local-exp} is controlled by the relative curvature pair \((\Hess f(x^\star),\Hess g(x^\star))\): a better-conditioned metric \(G(x^\star)\) yields faster local dissipation. This observation can be sharpened into a concrete design principle for DC decompositions.

\begin{proposition}
\label{prop:metric-quality}
Let \(x^\star\) be a critical point with \(\Hess f(x^\star)\succ 0\), and write
\[
H_f:=\Hess f(x^\star),
\qquad
G_\star:=\Hess g(x^\star).
\]
The linearization of the continuous DCA flow \eqref{eq:continuous-dca} at \(x^\star\) is
\begin{equation}
\label{eq:linearized-flow}
\dot z = -G_\star^{-1}H_f z.
\end{equation}
Consequently, the local exponential rate of the linearized flow is governed by the spectrum of \(G_\star^{-1}H_f\). In particular,
\[
\lambda_{\mathrm{lin}}=\lambda_{\min}(G_\star^{-1}H_f)
\]
is the worst-case asymptotic rate. If, moreover, \(G_\star=\alpha H_f\) for some \(\alpha>0\), then all linearized modes contract at the common rate \(\alpha^{-1}\).
\end{proposition}

\begin{proof}
Let
\[
F(x):=-\Hess g(x)^{-1}\grad f(x),
\]
so that \eqref{eq:continuous-dca} is \(\dot x=F(x)\). We differentiate \(F\) at the equilibrium \(x^\star\). By the product rule,
\[
\mathrm{D}F(x)
=-\mathrm{D}\bigl(\Hess g(x)^{-1}\bigr)\,\grad f(x)
-\Hess g(x)^{-1}\Hess f(x).
\]
Since \(\grad f(x^\star)=0\), the first term vanishes at \(x^\star\), and therefore
\[
\mathrm{D}F(x^\star)
=-\Hess g(x^\star)^{-1}\Hess f(x^\star)
=-G_\star^{-1}H_f,
\]
which proves that the linearization is exactly \eqref{eq:linearized-flow}.

Since \(G_\star\succ0\) and \(H_f\succ0\), the matrix \(G_\star^{-1}H_f\) is similar to the symmetric positive-definite matrix
\[
G_\star^{-1/2}H_fG_\star^{-1/2},
\]
so its eigenvalues are real and strictly positive. Hence the linear system \eqref{eq:linearized-flow} is exponentially stable. This proves the rate claim.

If \(G_\star=\alpha H_f\) for some \(\alpha>0\), then \(G_\star^{-1}H_f=\alpha^{-1}I\), so every mode decays at the same rate \(\alpha^{-1}\).
\end{proof}

\begin{proposition}
\label{prop:damped-local-rate}
Let \(x^\star\) be a critical point with \(H_f:=\Hess f(x^\star)\succ0\), and for \(\eta\in(0,1]\) define the damped DCA map
\[
\Phi_\eta(x):=(\grad g)^{-1}\bigl((1-\eta)\grad g(x)+\eta\grad h(x)\bigr).
\]
Then the linearization of \(\Phi_\eta\) at \(x^\star\) is
\begin{equation}
\label{eq:damped-linearization}
z_{k+1}=\bigl(I-\eta G_\star^{-1}H_f\bigr)z_k.
\end{equation}
Moreover, every eigenvalue of \(G_\star^{-1}H_f\) lies in \((0,1]\), and therefore the local linearized contraction factor is
\[
q_{\mathrm{loc}}(\eta)
:=\rho\bigl(I-\eta G_\star^{-1}H_f\bigr)
=1-\eta\lambda_{\min}(G_\star^{-1}H_f).
\]
In particular, over \(\eta\in(0,1]\), the local linearized rate is optimized by the full-step choice \(\eta=1\).
\end{proposition}

\begin{proof}
Let \(H_h:=\Hess h(x^\star)\), so that \(H_f=G_\star-H_h\). Since \(h\) is convex, \(H_h\succeq0\). Differentiating \(\Phi_\eta\) at \(x^\star\) gives
\[
\mathrm{D}\Phi_\eta(x^\star)
=G_\star^{-1}\bigl((1-\eta)G_\star+\eta H_h\bigr)
=I-\eta G_\star^{-1}(G_\star-H_h)
=I-\eta G_\star^{-1}H_f,
\]
which proves \eqref{eq:damped-linearization}. Because \(H_h\succeq0\), one has
\[
0\prec H_f=G_\star-H_h\preceq G_\star.
\]
Hence \(G_\star^{-1}H_f\) is similar to the symmetric positive-definite matrix \(G_\star^{-1/2}H_fG_\star^{-1/2}\), whose eigenvalues all lie in \((0,1]\). Therefore, for every eigenvalue \(\lambda\in(0,1]\) and every \(\eta\in(0,1]\),
\[
|1-\eta\lambda|=1-\eta\lambda.
\]
Taking the maximum over all eigenvalues yields
\[
\rho\bigl(I-\eta G_\star^{-1}H_f\bigr)
=\max_{\lambda\in\sigma(G_\star^{-1}H_f)}(1-\eta\lambda)
=1-\eta\lambda_{\min}(G_\star^{-1}H_f),
\]
which is strictly decreasing in \(\eta\). The optimal local linearized choice is therefore \(\eta=1\).
\end{proof}

Theorem~\ref{thm:damped-pl-rate} and Proposition~\ref{prop:damped-local-rate} show that there is no universal optimal relaxation parameter. The provable global rate guarantee derived from descent and a metric DC-PL inequality is strongest at \(\eta=\frac12\), whereas the local linearized rate near a nondegenerate local minimum is fastest at \(\eta=1\). In this sense, smaller relaxation improves regularization and global control, while larger relaxation favors aggressive local convergence.

\paragraph{Design principle.}
Proposition~\ref{prop:metric-quality} shows that DC decomposition quality is inseparable from metric quality. A convex component \(g\) does more than make \(h=g-f\) convex; it also supplies the Hessian metric that governs the local dynamics. Thus, near a target local minimum, it is reasonable to view \(\Hess g(x^\star)\) as a \emph{preconditioner} for \(\Hess f(x^\star)\). When \(G_\star\) is close to being proportional to \(H_f\), or more generally is well aligned with its eigenspaces without being unnecessarily large, the linearized dynamics is more isotropic and the local contraction is better conditioned. This gives a concrete guideline for decomposition design.

\section{Decomposition Sensitivity and Metric Design}
\label{sec:geometry}

The flow \eqref{eq:continuous-dca} depends on \(f\) only through the pair \((\grad f,G)\). Therefore two different decompositions of the same objective yield different vector fields whenever they induce different Hessian metrics. This section makes that point explicit.

\begin{proposition}
\label{prop:metric-comparison}
Let \(K\subset\R^n\) be compact, and assume \(x(t)\in K\) for all \(t\) in the interval of interest. If
\[
mI \preceq G(x) \preceq MI
\qquad \forall x\in K,
\]
then along the continuous DCA flow,
\begin{equation}
\label{eq:dissipation-bounds}
-\frac{1}{m}\norm{\grad f(x(t))}^2
\le \frac{\dd}{\dd t}f(x(t))
\le -\frac{1}{M}\norm{\grad f(x(t))}^2.
\end{equation}
\end{proposition}

\begin{proof}
From Proposition~\ref{prop:energy},
\[
\frac{\dd}{\dd t}f(x(t))
= -\grad f(x(t))^\top G(x(t))^{-1}\grad f(x(t)).
\]
The bounds
\(
mI \preceq G(x)\preceq MI
\)
imply, after inversion,
\(
M^{-1}I \preceq G(x)^{-1} \preceq m^{-1}I
\)
for every \(x\in K\). Hence, for \(v=\grad f(x(t))\),
\[
\frac{1}{M}\norm{v}^2
\le v^\top G(x(t))^{-1}v
\le \frac{1}{m}\norm{v}^2.
\]
Multiplying by \(-1\) and using the energy identity gives \eqref{eq:dissipation-bounds}.
\end{proof}

Proposition~\ref{prop:metric-comparison} shows that \(g\) modulates the dissipation rate through the inverse Hessian \(G^{-1}\). The convex part is therefore not only a modeling artifact; it is a \emph{preconditioner} encoded at the level of geometry.

The next example makes Proposition~\ref{prop:metric-comparison} concrete on a nonconvex objective: by changing the DC decomposition while keeping \(f\) fixed, one changes the state-dependent metric \(\Hess g\), and hence the resulting continuous DCA trajectory.

\begin{example}
\label{ex:doublewell}
Consider the nonconvex objective
\[
f(x)=\frac{1}{4}(x_1^4+x_2^4)-\frac{1}{2}(x_1^2+x_2^2),
\qquad x=(x_1,x_2)\in\R^2,
\]
which is the standard double-well potential in each coordinate. For any diagonal matrix
\[
Q=\mathrm{diag}(q_1,q_2),\qquad q_1,q_2>0,
\]
consider the DC decomposition
\[
\qquad
g_Q(x)=\frac{1}{4}(x_1^4+x_2^4)+\frac{1}{2}x^\top Qx,
\qquad
h_Q(x)=\frac{1}{2}x^\top(Q+I)x.
\]
Then \(f=g_Q-h_Q\), the function \(h_Q\) is convex, and
\[
\Hess g_Q(x)=
\begin{pmatrix}
3x_1^2+q_1 & 0 \\
0 & 3x_2^2+q_2
\end{pmatrix}
\succeq Q \succ 0,
\]
so \(g_Q\) is strongly convex. The continuous DCA flow therefore becomes
\[
\bigl(3x_1^2+q_1\bigr)\dot x_1 = x_1-x_1^3,
\qquad
\bigl(3x_2^2+q_2\bigr)\dot x_2 = x_2-x_2^3.
\]
Hence the same nonconvex objective generates different continuous dynamics as \(Q\) varies. For example, if \(q_1\neq q_2\) and \(x_1(0)=x_2(0)\in(0,1)\), then initially
\[
\dot x_1(0)=\frac{x_1(0)-x_1(0)^3}{3x_1(0)^2+q_1},
\qquad
\dot x_2(0)=\frac{x_1(0)-x_1(0)^3}{3x_1(0)^2+q_2},
\]
so the two coordinates evolve at different speeds. In particular, unless \(q_1=q_2\), the trajectory immediately leaves the diagonal \(x_1=x_2\). Thus even for a nonconvex objective with multiple wells, the choice of DC decomposition changes the continuous DCA path through the metric \(\Hess g_Q\).
\end{example}

Example~\ref{ex:doublewell} has an immediate design interpretation. Replacing \(g\) by \(g+\phi\) and \(h\) by \(h+\phi\), where \(\phi\) is convex and smooth, leaves the objective unchanged but replaces the metric \(G\) by \(G+\Hess\phi\). This creates a family of admissible flows indexed by convex regularizers \(\phi\). In that sense, choosing a DC decomposition is tantamount to choosing a Hessian metric in which steepest descent is performed.

\section{Bregman Geometry and Outlook on Nonsmooth Extensions}
\label{sec:nonsmooth}

This section has a more limited purpose than the preceding ones. First, it places the smooth continuous DCA flow into the Bregman-geometric language associated with the convex component \(g\). Second, it explains why the same dual-coordinate viewpoint points toward a nonsmooth extension, while also clarifying that a rigorous nonsmooth theory lies beyond the scope of the present work. The main contributions of the paper remain the smooth-flow derivation, its convergence analysis, and the resulting metric interpretation of DC decomposition quality.

The metric \(G=\Hess g\) is the infinitesimal form of the Bregman divergence generated by \(g\),
\[
D_g(z,x)=g(z)-g(x)-\ip{\grad g(x)}{z-x}.
\]
Indeed, for fixed \(x\) and small \(\xi\), Taylor expansion gives
\[
g(x+\xi)
=g(x)+\ip{\grad g(x)}{\xi}+\frac{1}{2}\xi^\top G(x)\xi+o(\norm{\xi}^2).
\]
Substituting this into the definition of \(D_g\) yields
\[
D_g(x+\xi,x) = \frac{1}{2}\xi^\top G(x)\xi + o(\norm{\xi}^2).
\]
Thus the quadratic form induced by the Hessian metric is precisely the second-order approximation of the Bregman divergence generated by \(g\). Equivalently, if one considers the local model
\[
\xi \mapsto \ip{\grad f(x)}{\xi} + \frac{1}{2}\xi^\top G(x)\xi,
\]
then its minimizer is
\[
\xi^\star = -G(x)^{-1}\grad f(x),
\]
which is exactly the instantaneous velocity prescribed by \eqref{eq:continuous-dca}. Therefore the continuous DCA flow is the local steepest-descent system associated with the Bregman geometry of the convex component. This is compatible with the recent observation that DCA admits a Bregman proximal interpretation \cite{FaustFawziSaunderson2023,Eckstein1998}.

The dual ODE \eqref{eq:dual-ode} sharpens this connection. Writing \(x=\grad g^*(y)\), where \(g^*\) is the Fenchel conjugate of \(g\), the flow becomes
\begin{equation}
\label{eq:dual-flow-conjugate}
\dot y = \grad h(\grad g^*(y)) - y,
\qquad
x = \grad g^*(y).
\end{equation}
This is a mirror-type relaxation in the \(g\)-dual coordinate: the state variable is no longer the primal point \(x\), but the dual coordinate \(y=\grad g(x)\), and the map \(\grad g^*\) transports the dynamics back to the primal space. Classical DCA is the full-step Euler discretization of \eqref{eq:dual-flow-conjugate}; the damped variant introduced in this paper is its small-step discretization. This dual formulation provides a direct bridge between discrete DCA, Bregman geometry, and continuous metric dynamics.

This also connects the present Bregman discussion with the decomposition-quality theme developed earlier. Changing the convex component \(g\) changes not only the Hessian metric \(G=\Hess g\), and hence the local and global rate mechanisms identified in Section~\ref{sec:analysis}, but also the underlying Bregman divergence \(D_g\) that defines the local geometry of the dual and primal formulations. From this perspective, the quality of a DC decomposition is simultaneously a metric-design problem and a Bregman-geometry selection problem: one is choosing both the metric in which the flow dissipates and the divergence in which the associated local model is formed.

For nonsmooth DC problems, the Hessian metric is unavailable, but the dual-coordinate picture still suggests plausible differential inclusions. If \(g\) remains Legendre while \(h\) is merely proper, convex, and lower semicontinuous, one is naturally led to
\begin{equation}
\label{eq:nonsmooth-dual}
\dot y(t) + y(t) \in \partial h(\grad g^*(y(t))),
\qquad
x(t)=\grad g^*(y(t)).
\end{equation}
Equivalently, in primal variables one expects a Bregman-gradient or primal-dual inclusion rather than a classical ODE. This observation aligns with Bregman proximal theory \cite{Eckstein1998} and with recent nonsmooth DC splitting methods \cite{BanertBot2019}. However, turning \eqref{eq:nonsmooth-dual} into a genuine theory requires several ingredients that are not needed in the smooth setting: one must choose an appropriate notion of solution, establish well-posedness for the resulting differential inclusion, control the regularity of the dual-primal map through \(g^*\), and recover an analogue of the exact DCA descent mechanism at the continuous level. These issues are substantial enough to warrant a separate treatment. Accordingly, \eqref{eq:nonsmooth-dual} is not presented here as a theorem of the paper, but as a mathematically natural candidate for future work suggested by the smooth theory developed here.

\section{Conclusion}
\label{sec:conclusion}

This paper shows that DCA admits a natural continuous-time representation once it is written in the appropriate coordinates. The damped variant introduced here plays two roles: for fixed \(\eta\in(0,1)\), it defines a Bregman-regularized DCA step with monotone descent, asymptotic stationarity, KL convergence under standard boundedness assumptions, and a global linear rate under a metric DC-PL condition; as \(\eta\downarrow 0\), it provides the small-step bridge from the discrete algorithm to the limiting flow. It should therefore be viewed not only as a technical device for the continuous limit, but also as a discrete DCA variant with its own convergence theory. In the dual variable \(y=\grad g(x)\), DCA is a fixed-point method and, at the same time, a full-step Euler discretization of a nonlinear autonomous ODE. Pulling the ODE back to primal coordinates yields a Hessian-Riemannian gradient system whose metric is generated by the convex part of the decomposition. This perspective accounts for descent, asymptotic stationarity, global rate estimates under metric relative error bounds, KL convergence, local exponential stability, and decomposition sensitivity within a single framework. It also shows that there is no universal optimal relaxation parameter: the strongest provable global metric-PL guarantee furnished by the damped-scheme analysis is obtained by the half-relaxed scheme, whereas the best local linearized rate near a nondegenerate minimum is obtained by the full-step scheme.

From the viewpoint of DCA theory, the value of this continuous-time picture is not only that it produces a new flow, but also that it clarifies why several familiar convergence properties of DCA hold. The objective becomes an explicit Lyapunov function, criticality of limit points follows from metric dissipation, metric relative error bounds appear as the natural global counterparts of Euclidean error-bound and Polyak-Lojasiewicz conditions, and KL geometry upgrades asymptotic criticality to full-trajectory convergence exactly as in the modern discrete theory. In this way, the flow organizes the convergence analysis of DCA in geometric terms rather than simply restating it in continuous time.

The same viewpoint also turns the long-recognized but poorly formalized issue of DC decomposition quality into a geometric question. The convex component \(g\) does not merely certify a valid splitting; it supplies the metric in which the dynamics evolves. The local analysis shows that this metric directly controls the linearized rate through the spectrum of \( \Hess g(x^\star)^{-1}\Hess f(x^\star)\). Thus a high-quality DC decomposition should be understood as one that convexifies the model while inducing a metric well aligned with the local curvature of \(f\). This gives a concrete guideline for decomposition design and suggests that metric selection is a natural language for further progress on decomposition quality in DCA.

The present analysis also suggests several continuations. One is to move from the global-versus-local parameter tradeoff identified here toward adaptive relaxation and optimal metric-selection rules for constructing high-quality DC decompositions. Beyond the smooth setting treated here, two further directions are accelerated or inertial DC flows whose discretizations remain algorithmically meaningful, and a rigorous theory for nonsmooth Legendre-Bregman differential inclusions capable of preserving the characteristic DCA descent structure.

\section*{Acknowledgments}
The author was supported by the National Natural Science Foundation of China (Grant No.~42450242) and the Beijing Overseas High-Level Talent Program. The author gratefully acknowledges institutional support from the Beijing Institute of Mathematical Sciences and Applications (BIMSA).

\section*{Data availability}
No datasets were generated for this theoretical study.

\section*{Conflict of interest}
The author declares no conflict of interest.


\begin{thebibliography}{10}
\providecommand{\url}[1]{{#1}}
\providecommand{\urlprefix}{URL }
\expandafter\ifx\csname urlstyle\endcsname\relax
  \providecommand{\doi}[1]{DOI~\discretionary{}{}{}#1}\else
  \providecommand{\doi}{DOI~\discretionary{}{}{}\begingroup
  \urlstyle{rm}\Url}\fi

\bibitem{AbbaszadehpeivastiDeKlerkZamani2024}
Abbaszadehpeivasti, H., de~Klerk, E., Zamani, M.: On the rate of convergence of
  the difference-of-convex algorithm ({DCA}).
\newblock Journal of Optimization Theory and Applications \textbf{202},
  475--496 (2024).
\newblock Https://doi.org/10.1007/s10957-023-02199-z

\bibitem{AbsilMahonySepulchre2008}
Absil, P.A., Mahony, R., Sepulchre, R.: Optimization Algorithms on Matrix
  Manifolds.
\newblock Princeton University Press, Princeton, NJ (2008)

\bibitem{AttouchBolteSvaiter2013}
Attouch, H., Bolte, J., Svaiter, B.F.: Convergence of descent methods for
  semi-algebraic and tame problems: proximal algorithms, forward-backward
  splitting, and regularized {Gauss-Seidel} methods.
\newblock Mathematical Programming \textbf{137}(1--2), 91--129 (2013).
\newblock Https://doi.org/10.1007/s10107-011-0484-9

\bibitem{BanertBot2019}
Banert, S., Bo{\c{t}}, R.I.: A general double-proximal gradient algorithm for
  d.c. programming.
\newblock Mathematical Programming \textbf{178}, 301--326 (2019).
\newblock Https://doi.org/10.1007/s10107-018-1292-2

\bibitem{BolteDaniilidisLewisShiota2007}
Bolte, J., Daniilidis, A., Lewis, A., Shiota, M.: Clarke subgradients of
  stratifiable functions.
\newblock SIAM Journal on Optimization \textbf{18}(2), 556--572 (2007).
\newblock Https://doi.org/10.1137/060670080

\bibitem{Deuflhard2011}
Deuflhard, P.: Newton Methods for Nonlinear Problems: Affine Invariance and
  Adaptive Algorithms.
\newblock Springer, Berlin, Heidelberg (2011)

\bibitem{Eckstein1998}
Eckstein, J.: Approximate iterations in {Bregman}-function-based proximal
  algorithms.
\newblock Mathematical Programming \textbf{83}, 113--123 (1998).
\newblock Https://doi.org/10.1007/BF02680553

\bibitem{FaustFawziSaunderson2023}
Faust, O., Fawzi, H., Saunderson, J.: A {Bregman} divergence view on the
  difference-of-convex algorithm.
\newblock In: Proceedings of The 26th International Conference on Artificial
  Intelligence and Statistics, \emph{Proceedings of Machine Learning Research},
  vol. 206, pp. 3427--3439 (2023).
\newblock Available at https://proceedings.mlr.press/v206/faust23a.html

\bibitem{KarimiNutiniSchmidt2016}
Karimi, H., Nutini, J., Schmidt, M.: Linear convergence of gradient and
  proximal-gradient methods under the {Polyak-Lojasiewicz} condition.
\newblock In: Machine Learning and Knowledge Discovery in Databases, pp.
  795--811. Springer (2016)

\bibitem{KricheneBayenBartlett2015}
Krichene, W., Bayen, A., Bartlett, P.L.: Accelerated mirror descent in
  continuous and discrete time.
\newblock In: Advances in Neural Information Processing Systems 28, pp.
  2845--2853 (2015)

\bibitem{Kurdyka1998}
Kurdyka, K.: On gradients of functions definable in o-minimal structures.
\newblock Annales de l'Institut Fourier \textbf{48}(3), 769--783 (1998).
\newblock Https://doi.org/10.5802/aif.1638

\bibitem{LeThiPham2018}
Le~Thi, H.A., Pham, D.T.: Dc programming and {DCA}: thirty years of
  developments.
\newblock Mathematical Programming \textbf{169}(1), 5--68 (2018).
\newblock Https://doi.org/10.1007/s10107-018-1235-y

\bibitem{LeThiDinhPham2018}
Le~Thi, H.A., Pham, D.T., Pham, T.V.: Convergence analysis of
  difference-of-convex algorithm with subanalytic data.
\newblock Journal of Optimization Theory and Applications \textbf{179}(1),
  103--126 (2018).
\newblock Https://doi.org/10.1007/s10957-018-1345-y

\bibitem{Niu2022}
Niu, Y.S.: On the convergence analysis of {DCA} (2022).
\newblock \urlprefix\url{https://arxiv.org/abs/2211.10942}

\bibitem{PhamLeThi1997}
Pham, D.T., Le~Thi, H.A.: Convex analysis approach to dc programming: theory,
  algorithms and applications.
\newblock Acta Mathematica Vietnamica \textbf{22}(1), 289--355 (1997)

\bibitem{PhamLeThi1998}
Pham, D.T., Le~Thi, H.A.: A dc optimization algorithm for solving the
  trust-region subproblem.
\newblock SIAM Journal on Optimization \textbf{8}(2), 476--505 (1998).
\newblock Https://doi.org/10.1137/S1052623493254025

\bibitem{Polyak1963}
Polyak, B.T.: Gradient methods for minimizing functionals.
\newblock USSR Computational Mathematics and Mathematical Physics
  \textbf{3}(4), 864--878 (1963).
\newblock Https://doi.org/10.1016/0041-5553(63)90382-3

\bibitem{SuBoydCandes2016}
Su, W., Boyd, S., Cand{\`e}s, E.J.: A differential equation for modeling
  {Nesterov}'s accelerated gradient method: Theory and insights.
\newblock Journal of Machine Learning Research \textbf{17}(153), 1--43 (2016)

\end{thebibliography}
\end{document}